\documentclass[12pt, twoside]{article}
\usepackage{enumitem}
\usepackage{tocloft}

\usepackage{titletoc}
\usepackage{pdfpages}
\usepackage{amsmath,amssymb}
\usepackage{amsthm}
\usepackage{leftidx}
\usepackage{epsfig}
\usepackage{graphics}
\usepackage{graphicx}
\usepackage{colordvi}
\usepackage{mathrsfs} 
\usepackage{stmaryrd}
\usepackage{geometry}
\usepackage{dsfont}
\usepackage{tikz}
\usetikzlibrary{automata}
\usetikzlibrary{snakes}
\usetikzlibrary{decorations.pathmorphing}
\usepackage{url}
\usepackage[hidelinks] 
{hyperref}
\hypersetup{
    colorlinks=false,
   linkcolor=blue!75!black,
    filecolor=blue!75!black,      
   urlcolor=blue!75!black, 
   citecolor=blue!75!black,
}
\usepackage{enumitem}
\usepackage[textsize=small]{todonotes}
\usepackage[margin=1cm, size=small]{caption}
\usepackage{multirow}
\allowdisplaybreaks[4]
\oddsidemargin
-.25in \evensidemargin -.25in\marginparwidth 0.4in
\usepackage{color,soul}
\usepackage{xcolor}

\usepackage{etoolbox}
\usepackage{fancyvrb}
\usepackage{fvextra}
\usepackage{fancyhdr}
\usepackage{todonotes}

\usepackage{xr}
\numberwithin{figure}{section}

\usepackage{empheq}
\usetikzlibrary{automata}
\definecolor{blue2}{cmyk}{.94,.11,0,0}
\usepackage{enumitem}
\setitemize{itemsep=0pt}
\definecolor{myblue}{rgb}{.8, .8, 1}
\newlength\mytemplen
\newsavebox\mytempbox
\usepackage{verbatim}

\DeclareFontFamily{U}{mathx}{}
\DeclareFontShape{U}{mathx}{m}{n}{<-> mathx10}{}
\DeclareSymbolFont{mathx}{U}{mathx}{m}{n}
\DeclareMathAccent{\widecheck}{0}{mathx}{"71}
\DeclareMathAlphabet\mathbfcal{OMS}{cmsy}{b}{n}

\makeatletter
\renewenvironment{thebibliography}[1]
     {\section*{\refname}
      \@mkboth{\MakeUppercase\refname}{\MakeUppercase\refname}
      \begin{enumerate}[label={[\arabic{enumi}]},itemindent=*,leftmargin=2.5em]
      \@openbib@code
      \sloppy
      \clubpenalty4000
      \@clubpenalty \clubpenalty
      \widowpenalty4000
      \sfcode`\.\@m}
     {\def\@noitemerr
       {\@latex@warning{Empty `thebibliography' environment}}
      \end{enumerate}}
\makeatother

\makeatletter
\newcommand{\newparallel}{\mathrel{\mathpalette\new@parallel\relax}}
\newcommand{\new@parallel}[2]{
  \begingroup
  \sbox\z@{$#1T$}
  \resizebox{!}{\ht\z@}{\raisebox{\depth}{$\m@th#1/\mkern-4.5mu/$}}
  \endgroup
}
\makeatother

\makeatletter
\newcommand{\nnewparallel}{\mathrel{\mathpalette\nnew@parallel\relax}}
\newcommand{\nnew@parallel}[2]{
  \begingroup
  \sbox\z@{$#1T$}
  \resizebox{!}{\ht\z@}{\raisebox{\depth}{$\m@th#1/\mkern-4.5mu/\!\!\!\!\backslash $}}
  \endgroup
}
\makeatother

\makeatletter
\newcommand\cb{
    \@ifnextchar[
       {\@cb}
       {\@cb[5pt]}}

\def\@cb[#1]{
    \@ifnextchar[
       {\@@cb[#1]}
       {\@@cb[#1][5pt]}}

\def\@@cb[#1][#2]#3{
    \sbox\mytempbox{#3}
    \mytemplen\ht\mytempbox
    \advance\mytemplen #1\relax
    \ht\mytempbox\mytemplen
    \mytemplen\dp\mytempbox
    \advance\mytemplen #2\relax
    \dp\mytempbox\mytemplen
    \colorbox{myblue}{\hspace{1em}\usebox{\mytempbox}\hspace{1em}}}
\makeatother

\newcommand{\lmt}{\longmapsto}

\newcommand{\la}{{\langle}}
\newcommand{\ra}{{\rangle}}

\newcommand{\ind}{{\perp\!\!\!\perp}}

\newcommand{\bs}{\boldsymbol}

\newcommand{\mc}{\mathcal}
\renewcommand{\d}{{\mathrm d}}

\newcommand{\R}{{\Bbb R}}

\newcommand{\Id}{{\rm Id}}

\newcommand{\defeq}{{\stackrel{\rm def}{=}}}

\renewenvironment{proof}[1][\proofname]{\noindent {\bfseries #1.}\;}{\hfill\ensuremath{\blacksquare}\\}

\newcommand{\cc}{{^\circ}}

\newcommand{\bi}{\mathbf i}
\newcommand{\bj}{\mathbf j}
\newcommand{\bk}{\mathbf k}

\newcommand{\two}{{\sqrt{2}}}

\newcommand{\CN}{E^{\it N}}

\newcommand{\wt}{\widetilde}

\patchcmd{\thebibliography}{\section*}{\section}{}{}

\newtheoremstyle{slantthm}{10pt}{10pt}{\slshape}{}{\bfseries}{}{.5em}{\thmname{#1}\thmnumber{ #2}\thmnote{ (#3)}.}
\newtheoremstyle{slantrmk}{10pt}{10pt}{\rmfamily}{}{\bfseries}{}{.5em}{\thmname{#1}\thmnumber{ #2}\thmnote{ (#3)}.}

\begin{document}
\theoremstyle{slantthm}
\newtheorem{thm}{Theorem}[section]
\newtheorem{prop}[thm]{Proposition}
\newtheorem{lem}[thm]{Lemma}
\newtheorem{cor}[thm]{Corollary}
\newtheorem{defi}[thm]{Definition}
\newtheorem{claim}[thm]{Claim}
\newtheorem{disc}[thm]{Discussion}
\newtheorem{conj}[thm]{Conjecture}

\theoremstyle{slantrmk}
\newtheorem{ass}[thm]{Assumption}
\newtheorem{rmk}[thm]{Remark}
\newtheorem{eg}[thm]{Example}
\newtheorem{question}[thm]{Question}
\numberwithin{equation}{section}
\newtheorem{quest}[thm]{Quest}
\newtheorem{problem}[thm]{Problem}
\newtheorem{discussion}[thm]{Discussion}
\newtheorem{notation}[thm]{Notation}
\newtheorem{observation}[thm]{Observation}

\definecolor{db}{RGB}{13,60,150}
\definecolor{dg}{RGB}{150,40,40}

\newcommand{\thetitle}{\vspace{-1.5cm}
Stochastic motions of the two-dimensional many-body delta-Bose gas, IV: Transformations of relative motions}

\title{\bf \thetitle\footnote{Support from an NSERC Discovery grant is gratefully acknowledged.}}

\author{Yu-Ting Chen\,\footnote{Department of Mathematics and Statistics, University of Victoria, British Columbia, Canada.}\,\,\footnote{Email: \url{chenyuting@uvic.ca}}}

\date{\today\vspace{-1cm}}

\maketitle
\abstract{This paper is the last in a series devoted to constructing stochastic motions representing the two-dimensional $ N$-body delta-Bose gas for all integers $N\geq 3$ via Feynman--Kac-type formulas. The main result here supplements \cite{C:SDBG1-4,C:SDBG2-4} of the series by proving a bijective transformation between two general classes of Langevin-type SDEs such that the SDEs of one class describe precisely the stochastic relative motions of the SDEs of the other class.  \medskip 

\noindent \emph{Keywords:} Delta-Bose gas; Schr\"odinger operators; interacting diffusions. \smallskip

\noindent \emph{Mathematics Subject Classification (2020):} 60H10; 60G07; 35Q40.\vspace{-.4cm}
}

\section{Introduction}\label{sec:intro}
This paper is the last in a series devoted to constructing stochastic motions representing the two-dimensional many-body delta-Bose gas via Feynman--Kac-type formulas. Our goal here is to study two closely related problems on stochastic descriptions in the sense of E. Nelson \cite{Nelson-4} for general many-body quantum Hamiltonians of the following form:
\begin{align}\label{def:QMHamiltonian}
-\frac{1}{2}\sum_{i=1}^N\Delta_{z^i}+\sum_{\bi\in \mathcal E_N} V_\bi(z^{i\prime}-z^i),\quad z^i\in E,
\end{align}
where $E\,\defeq\, \R^d$ for $d\in \Bbb N$, $N\geq 2$, $\mathcal E_N\,\defeq\,\{\bi=(i\prime,i)\in \Bbb N^2;1\leq i<i\prime\leq N\}$, $\Delta_z$ denotes the $d$-dimensional Laplacian in $z\in \R^d$, and $V_\bi(\cdot)$ are meant as potential energy functions. 

The two problems mentioned above concern the transformations of SDEs. Specifically, the first problem considers the following Langevin-type SDEs:
\begin{align}\label{def:Ziintro}
\mc Z^i_t=\mc Z^i_0-\sum_{j:j\neq i}\int_0^t a_{i j}(s)(\mc Z^i_s-\mc Z^j_s)\d s+W^i_t,\quad \forall\;1\leq i\leq N,
\end{align}
obeyed by a family of $E$-valued processes $\{\mathcal Z^i_t\}_{1\leq i\leq N}$ for $N$ particles. 
Here, $a_{ij}(s)$ are progressively measurable real-valued processes such that integral terms in \eqref{def:Ziintro} are well-defined, and $\{W^i_t\}_{1\leq i\leq N}$ consists of independent $d$-dimensional standard Brownian motions with zero initial conditions. Under this setting, we are interested in the transformations of the SDEs in \eqref{def:Ziintro} to the SDEs obeyed by the {\bf stochastic relative motions} defined by the following $E$-valued processes:
 \begin{align}\label{def:Zpbia}
 (\mathcal Z^{i\prime}_t-\mathcal Z^i_t)/\two,\quad \bi=(i\prime,i)\in \mc E_N.
 \end{align}
Moreover, we seek to determine the SDEs of the corresponding transformation in \eqref{def:Zpbia}
to the degree of expressing the drift coefficients of the SDEs \emph{in terms of} the stochastic relative motions. 
By contrast, the second problem is a converse on reconstructing the SDEs of $E$-valued processes $\{\mathcal Z^i_t\}_{1\leq i\leq N}$ for $N$ particles from given SDEs obeyed by a family of $E$-valued processes $\{\mathcal Z^\bi_t\}_{\bi\in \mc E_N}$ such that $\{\mathcal Z^\bi_t\}_{\bi\in \mc E_N}$ recovers the stochastic relative motions \eqref{def:Zpbia} of $\{\mathcal Z^i_t\}_{1\leq i\leq N}$. Note that these two problems have arisen from \cite{C:SDBG1-4,C:SDBG2-4} for the quantum Hamiltonian of the two-dimensional many-delta Bose gas, whose formal description uses \eqref{def:QMHamiltonian} with $V_\bi(z)\equiv -\Lambda_\bi\delta(z)$ and $d=2$. In particular, the SDEs in \eqref{def:Ziintro} generalize the SDEs of the stochastic one-$\delta$ and many-$\delta$ motions studied in \cite{C:SDBG1-4,C:SDBG2-4}. See Remark~\ref{rmk:motivation} for additional details. 

The main result of this paper (Proposition~\ref{prop:RM}) solves 
Problem~\ref{prob:RM} stated below, which summarizes the two problems mentioned above and begins with a preliminary problem concerning the ``consistency'' of states that is crucial for solving the second problem.

\begin{problem}\label{prob:RM}
Let $N\geq 2$ be an integer.
\begin{itemize}
\item [\rm (1$\cc$)] Determine minimal conditions on families $\{z^\bi\}_{\bi\in \mc E_N}$ of $E$-valued points such that these families equal the families of states of (normalized) relative motions
\begin{align}\label{def:zbia}
(z^{i\prime}-z^i)/\two, \quad \bi=(i\prime,i)\in \mc E_N,
\end{align}
for some $\{z^i\}_{1\leq i\leq N}$ of $E$-valued points regarded as the states of $N$ particles in $E$.

\item [\rm (2$\cc$)] Assume that $\{\mathcal Z^i_t\}_{1\leq i\leq N}$ obeys \eqref{def:Ziintro}. Determine the SDEs of the associated stochastic relative motions such that the drift coefficients are expressed in terms of the stochastic relative motions and $a_{ij}(s)$.  
\item [\rm (3$\cc$)] Determine minimal conditions on SDEs with solutions $\{\mathcal Z_t^\bi\}_{\bi\in \mc E_N}$ of $E$-valued processes such that these solutions coincide with the stochastic relative motions of solutions of SDEs taking the form of \eqref{def:Ziintro}. 
\end{itemize}
\end{problem}

More specifically, Problem~\ref{prob:RM}~(1$\cc$) asks to solve for $\{z^i\}_{1\leq i\leq N}$ from the linear equations $z^\bi=(z^{i\prime}-z^i)/\two$ for all $\bi=(i\prime,i)\in \mc E_N$. The issue here is that for $N=3$, the $3\times 3$ coefficient matrix of these linear equations is not invertible. Moreover, as soon as $N\geq 4$, the system of these linear equations is overdetermined. 

In the remainder of this paper, {\bf we neither consider only $d=2$ nor follow the notations in the former parts \cite{C:SDBG1-4,C:SDBG2-4,C:SDBG3-4} of the series in which $\mathcal Z^\bi_t$'s denote the processes in (\ref{def:Zpbia}) and $z^\bi$'s denote the states in (\ref{def:zbia}).}  Unless otherwise mentioned, $\bi$'s in the superscripts will only mean to index processes or states. Moreover, the proofs in this paper are self-contained and do not depend on any former parts of the series. \vspace{-.4cm}

\section{Transformations of Langevin-type SDEs}
Our goal in this section is to solve Problem~\ref{prob:RM}. The key notion we introduce to approach Problem~\ref{prob:RM} (1$\cc$) is the following definition. 

\begin{defi}\label{def:dc}
Fix an integer $N\geq 2$. A family $\{z^\bi\}_{\bi\in \mathcal E_N}$ of $E$-valued points is {\bf difference-consistent} if for any $\bi=(i\prime,i),\bj=(j\prime,j)\in \mathcal E_N$ with $i=j\prime$, 
\begin{align}\label{defeq:dc}
z^\bi+z^\bj=z^{\bi\oplus \bj},\quad \mbox{ where }\bi\oplus \bj=\bj\oplus \bi\,\defeq\,(i\prime,j)\in \mathcal E_N.
\end{align}
The set of $(z^{(2,1)},\cdots,z^{(N,N-1)})$ from
 difference-consistent families is denoted by $\delta \CN$.
\end{defi}

Definition~\ref{def:dc} specifies the first minimal condition for solving Problem~\ref{prob:RM} (1$\cc$).
Note that for $N=2$, $\{z^\bi\}_{\bi\in \mc E_2}$ is a singleton, so the difference-consistency holds trivially by definition. For any $N\geq 3$, 
any difference-consistent family $\{z^\bi\}_{\bi\in \mathcal E_N}$ is uniquely determined by $(z^{(2,1)},\cdots,z^{(N,N-1)})$ via the following {\bf telescoping representation}: 
\begin{align}\label{eq:telescoping}
\forall\;\bi=(i\prime,i)\in \mathcal E_N, \; z^{(i\prime,i)}=z^{(i\prime,i\prime-1)}+\cdots +z^{(i+1,i)}. 
\end{align}
In particular, since, given any $(z^{(2,1)},\cdots,z^{(N,N-1)})\in E^{N-1}$, \eqref{eq:telescoping} defines a difference-consistent family, the linear space $\delta \CN$ is isomorphic to $E^{N-1}$ as vector spaces.  

Proposition~\ref{prop:RM} below is the main result of this paper. Proposition~\ref{prop:RM}~(1$\cc$) gives a complete solution to Problem~\ref{prob:RM} (1$\cc$), and Proposition~\ref{prop:RM}~(2$\cc$) solves Problem~\ref{prob:RM} (2$\cc$) and (3$\cc$) simultaneously with a bijective transformation between two general classes of Langevin-type SDEs. Via this transformation, the SDEs of one class describe precisely the stochastic relative motions of the SDEs of the other class. In particular, Definition~\ref{def:dc} and the center of mass implied by the last row of $\mathsf R$ in \eqref{def:R} are the key ingredients of our resolution of Problem~\ref{prob:RM}; see Remark~\ref{rmk:com} for more details. We remark that for $N\geq 3$, our choice of the center of mass to complement Definition~\ref{def:dc} was originally motivated by \cite{SP:08-4} in a similar context of interacting Brownian motions, while the use of the center of mass is classical for $N=2$. See also \cite{GH:Short-4} on applying the center of mass at the operator level for $N\geq 3$. 

\begin{prop}[Main result]\label{prop:RM}
{\rm (1$\cc$)} As vector spaces,
$\CN$ is isomorphic to  $\delta \CN\times E$ via the following
linear transformation $\mathsf R:\CN\to \delta \CN\times E$:
\begin{align}
\mathsf R\;\defeq\, \begin{bmatrix}
-1/\two & 1/\two & 0&\cdots &0\\
0 &-1/\two&1/\two& \cdots &0\\
\vdots &\vdots &\ddots &\vdots&\vdots\\
0 &0 &\cdots &-1/\two &1/\two\\
1/N&1/N &\cdots &1/N &1/N
\end{bmatrix}_{N\times N}:
\begin{bmatrix}
z^1\\
z^2\\
\vdots\\
z^N
\end{bmatrix}
\lmt \begin{bmatrix}
z^{(2,1)}\\
\vdots\\
z^{(N,N-1)}\\
z^\Sigma
\end{bmatrix}.\label{def:R}
\end{align}
Specifically, $\mathsf R$ is invertible, and the matrix inverse admits the following formula:
\begin{align}\label{def:Rinv}
(\mathsf R^{-1})_{ij}=
\begin{cases}
\displaystyle -\two (N-j)/N,&\; i\leq j,\; j\leq N-1,\\
\displaystyle \two j/N,&\;i>j,\;j\leq N-1,\\
\displaystyle 1, &\;1\leq i\leq N,\;j=N.
\end{cases}
\end{align} 

\noindent {\rm (2$\cc$)} There is a pathwise one-to-one correspondence between the two classes [$N$] and [$\mathcal E_N$] of families of processes $\{\mc Z^i_t \}_{1\leq i\leq N}$ and $\{\mathcal Z^\bi_t \}_{\bi\in\mc E_N}\cup \{\mc Z^\Sigma_0+W^\Sigma_t \}$ defined as follows.
\begin{itemize}[leftmargin=1.2cm]
\item [{\bf [$\bs N$]}] The processes $\{\mc Z^i_t \}$  satisfy the following SDEs:
\begin{align}\label{def:Zi}
\mc Z^i_t=\mc Z^i_0-\sum_{j:j\neq i}\int_0^t a_{i j}(s)(\mc Z^i_s-\mc Z^j_s)\d s+W^i_t,\quad \forall\;1\leq i\leq N,
\end{align}
such that the following conditions hold:
\begin{itemize}
\item [$\bullet$] For all $i\neq j$, $a_{ij}(s)=a_{ij}(s,\omega)$ is progressively measurable, the symmetry $a_{ij}=a_{ji}$ holds, and
$\int_0^t |a_{ij}(s)(\mc Z^i_s-\mc Z^j_s)|\d s<\infty$.
\item [$\bullet$]  $\{W^i_t \}_{1\leq i\leq N}$ consists of independent $d$-dimensional standard Brownian motions starting from $0$.
\end{itemize}

\item [{\bf [$ \mathbfcal{E}_{\bs N}$]}] The processes $\{\mc Z^\bi_t \}$ satisfy the following SDEs:
\begin{align}\label{def:tZ}
\mc Z^\bi_t=\mc Z^\bi_0-\sum_{\bj\in \mc E_N}\sigma(\bi)\cdot \sigma(\bj)\int_0^t a_\bj(s)\mc Z^\bj_s\d s+W^{\bi}_t,\quad \forall\;\bi\in \mc E_N,
\end{align}
$\{\sqrt{N}W^\Sigma_t \}$ is a $d$-dimensional standard Brownian motion starting from $0$, and $\mc Z^\Sigma_0\in E$, such that the following conditions hold:
\begin{itemize}
\item [$\bullet$] $\{\mc Z^\bi_0\}_{\bi\in \mc E_N}$ is difference-consistent.
\item [$\bullet$] $\sigma(\bi)\in \{-1,0,1\}^N$ is such that the $i\prime$-th component is $1$, the $i$-th component is $-1$, and all the other components are zero. Hence, $\sigma(\bi)\cdot\sigma(\bj)$ in \eqref{def:tZ} denotes the usual dot product of vectors.
\item [$\bullet$]  For all $\bi\in \mc E_N$, $a_{\bi}(s)=a_{\bi}(s,\omega)$ are progressively measurable, and we have $\int_0^t |a_{\bi}(s)\mc Z^\bi_s|\d s<\infty$.
\item [$\bullet$] $\{W^\bi_t \}_{\bi\in\mathcal E_N}$ satisfies the following properties. 
(i) It is independent of $\{W^\Sigma_t \}$; (ii) it has the same law of a linear transformation of Brownian motion such that $\{W^\bi_t \}$ is a $d$-dimensional standard Brownian motion with zero initial condition $0$ and real-valued components $\{W_t^{\bi,\ell}\}$, $1\leq \ell\leq d$, such that the following identities of quadratic variation hold:
$\la W^{\bi,\ell},W^{\bj,\ell}\ra_t=\sigma(\bi)\cdot \sigma(\bj)t/2$ and $
\la W^{\bi,k},W^{\bj,\ell}\ra_t=0$
for all $1\leq k, \ell\leq d$ with $k\neq \ell$ and all $\bi,\bj\in \mc E_N$.
\end{itemize}
\end{itemize}
Specifically, the one-to-one correspondence between class [$N$] and class [$\mathcal E_N$] is defined as follows: 
\begin{align}\label{def:bijective}
\begin{cases}
\displaystyle \frac{\mc Z_t^{i\prime}-\mc Z_t^i}{\two}=\mathcal Z_t^\bi,
\quad a_{ij}(s,\omega)=a_{i\vee\!\wedge j}(s,\omega),\quad \frac{W^{i\prime}_t-W^i_t}{\two}=W^{\bi}_t,\\
\vspace{-.4cm}\\
\displaystyle
\frac{1}{N}\sum_{i=1}^N \mc Z^i_t=\mc Z^\Sigma_0+W^\Sigma_t,\quad 
\frac{1}{N}\sum_{i=1}^N W^i_t=W^\Sigma_t,
\end{cases}
\end{align}
where $i\!\vee\!\!\wedge j=j\!\vee\!\!\wedge i\,\defeq\, 
(\max\{i,j\},\min\{i,j\})\in \mc E_N$ for any integers $1\leq i\neq j\leq N$.
\end{prop}

\begin{rmk}[Transformations of relative motions]\label{rmk:motivation}
We have studied and applied particular SDEs of class [$N$] extensively in the former parts of this series. Nevertheless, our initial method in preparing this series continued the spirit of using relative motions as in the classical method for the two-body delta-Bose gas and in the formulation of the possible alternative discussed in \cite[Section~2]{C:SDBG2-4}. This method led to formulating SDEs of class [$\mathcal E_N$] first and then seeking the corresponding solutions $\{\mathcal Z^i_t\}_{1\leq i\leq N}$.\qed 
\end{rmk}

For the proof of Proposition~\ref{prop:RM}, the main technical steps address the question of how the drift terms of $\{\mc Z^i_t\}_{1\leq i\leq N}$ transform to the drift terms of $\{\mc Z^\bi_t\}_{\bi\in \mc E_N}$, and conversely. In doing so, we nevertheless circumvent direct applications of the formula \eqref{def:Rinv} of $\mathsf R^{-1}$ in those transformations of drifts. Formula~\eqref{def:Rinv} is used here only to prove the invertibility of $\mathsf R$.

\begin{rmk}[Symmetry and stochastic center of mass]\label{rmk:com}
{\rm (1$\cc$)} The mapping between $a_{i\vee\!\wedge j}$ and $a_{ij}$ in \eqref{def:bijective} is well-defined because of the symmetry assumption $a_{ij}=a_{ji}$.\medskip 

\noindent {\rm (2$\cc$)} By $a_{ij}=a_{ji}$, \eqref{def:Zi} implies that the {\bf stochastic center of mass} $N^{-1}\sum_{i=1}^N \mc Z^i_t$ satisfies
\begin{align}\label{eq:com}
\frac{1}{N}\sum_{i=1}^N \mc Z^i_t=\frac{1}{N}\sum_{i=1}^N \mc Z^i_0+\frac{1}{N}\sum_{i=1}^NW^i_t,
\end{align}
and $\{\sqrt{N}\cdot N^{-1}\sum_{i=1}^NW^i_t \}$ is a $d$-dimensional standard Brownian motion independent of the correlated Brownian motions $\{(W^{i\prime}_t-W^i_t)/\two \}_{\bi\in \mc E_N}$. In particular, it is according to \eqref{eq:com} that we choose the last row of $\mathsf  R$ in \eqref{def:R}. 
\qed 
\end{rmk}

\begin{proof}[Proof of Proposition~\ref{prop:RM} (1$\cc$)]
To prove \eqref{def:Rinv}, we make two observations to simplify the algebra a bit.
First, observe that by \eqref{def:R},
\begin{align}\label{def:Q}
\mathsf R=\begin{bmatrix}
1/\two &0&\cdots &0\\
0 &1/\two&\cdots &0\\
\vdots &\vdots &\ddots &\vdots\\
0 &\cdots & 1/\two &0\\
0 &\cdots &0 &1/N
\end{bmatrix} 
\mathsf  Q,\quad \mbox{for }\mathsf Q\;\defeq\, \begin{bmatrix}
-1 & 1 & 0&0&\cdots &0\\
0 &-1&1& 0&\cdots &0\\
\vdots &\vdots &\vdots &\ddots &\vdots&\vdots\\
0 &0 &0&\cdots &-1 &1\\
1&1 &1&\cdots &1 &1
\end{bmatrix}.
\end{align}
Second, for any $N\times N$-matrix $\widetilde{\mathsf M}=[\widetilde{\mathsf M}_1,\cdots,\widetilde{\mathsf M}_N]$ and any scalars $D_1,\cdots,D_N$,
\begin{align}\label{MD}
[\widetilde{\mathsf M}_1,\cdots,\widetilde{\mathsf M}_N]\mathsf {\rm diag}(\mathsf D_1,\cdots,\mathsf D_N)=[\mathsf D_1\widetilde{\mathsf M}_1,\cdots,\mathsf D_N \widetilde{\mathsf M}_N].
\end{align}
Thus, the proposed inverse in \eqref{def:Rinv} takes the form of the left-hand side of \eqref{MD} with $\mathsf D_1=\cdots=\mathsf D_{N-1}=\sqrt{2}$, $\mathsf D_N=N$, and $\widetilde{\mathsf M}$ chosen to be $\mathsf M$ given by 
\begin{align}\label{def:Qinv}
\mathsf M_{ij}\defeq
\begin{cases}
\displaystyle -(N-j)/N,&\; i\leq j,\; j\leq N-1,\\
\displaystyle j/N,&\;i>j,\;j\leq N-1,\\
\displaystyle 1/N, &\;1\leq i\leq N,\;j=N.
\end{cases}
\end{align} 
Combining these two observations shows that it is enough to prove the invertibility of $\mathsf Q$ with $\mathsf M=\mathsf Q^{-1}$, or just the identity
$\mathsf Q\mathsf M=\Id $. 

Now, to verify $\mathsf Q\mathsf M=\Id $, we use the definitions of $\mathsf Q$ and $\mathsf M$
 in \eqref{def:Q} and \eqref{def:Qinv} and consider $(\mathsf Q\mathsf M)_{ij}=\sum_{k=1}^N \mathsf Q_{ik}\mathsf M_{kj}$, $1\leq i,j\leq N$, for five cases which are mutually exclusive, collectively exhaustive: \medskip 

\noindent {\bf (1)} For all $1\leq i< j\leq N$,
\begin{align*}
(\mathsf Q\mathsf M)_{ij}&
=\mathsf Q_{ii}\mathsf M_{ij}+\mathsf Q_{i(i+1)}\mathsf M_{(i+1)j}
=-\mathsf M_{ij}+\mathsf M_{(i+1)j}
=
\begin{cases}
\displaystyle \frac{N-j}{N}-\frac{N-j}{N}=0,&j\leq N-1,\\
\vspace{-.3cm}\\
\displaystyle -\frac{1}{N}+\frac{1}{N}=0,&j=N.
\end{cases}
\end{align*}

\noindent {\bf (2)} For all $1\leq j<i\leq N-1$,
\begin{align*}
(\mathsf Q\mathsf M)_{ij}&
=\mathsf Q_{ii}\mathsf M_{ij}+\mathsf Q_{i(i+1)}\mathsf M_{(i+1)j}=-\mathsf M_{ij}+\mathsf M_{(i+1)j}=-\frac{j}{N}+\frac{j}{N}=0.
\end{align*}

\noindent {\bf (3)} For $1\leq j<i=N$,
\begin{align*}
(\mathsf Q\mathsf M)_{Nj}
=\sum_{k=1}^N\mathsf M_{kj}=j \left(-\frac{N-j}{N}\right)+(N-j)\frac{j}{N}=0.
\end{align*}

\noindent {\bf (4)} For all $1\leq i=j\leq N-1$,
\begin{align*}
(\mathsf Q\mathsf M)_{ii}
=\mathsf Q_{ii}\mathsf M_{ii}+\mathsf Q_{i(i+1)}\mathsf M_{(i+1)i}=-\mathsf M_{ii}+\mathsf M_{(i+1)i}=\frac{N-i}{N}+\frac{i}{N}=1.
\end{align*}

\noindent (5) For $i=j=N$,
\begin{align*}
(\mathsf Q\mathsf M)_{NN}
=\sum_{k=1}^N\mathsf M_{kN}=\sum_{k=1}^N \frac{1}{N}=1.
\end{align*}
By the last equalities in these five cases,  we have proved that the matrix $\mathsf M$ defined by \eqref{def:Qinv} satisfies $\mathsf Q\mathsf M={\rm Id}$. The proof is complete.
\end{proof}

\vspace{-.2cm}

\begin{proof}[Proof of Proposition~\ref{prop:RM} (2$\cc$)]
We first show that given a family of processes $\{\mc Z^i_t \}_{1\leq i\leq N}$ from class [$N$], the family of processes $\{\mathcal Z^\bi_t \}_{\bi\in \mc E_N}\cup \{\mc Z^\Sigma_0+W^\Sigma_t \}$ defined by \eqref{def:bijective} is in class [$\mathcal E_N$]. It suffices to show \eqref{def:tZ}. 
To this end, first, we use the assumed equations \eqref{def:Zi} 
 to get
\begin{align}
\begin{aligned}\label{eq:dc0}
\frac{\mc Z^{i\prime}_t-\mc Z^i_t}{\two}&=\frac{\mc Z^{i\prime}_0-\mc Z^i_0}{\two}-\sum_{j_1:j_1\neq i\prime}\int_0^t a_{i\prime\vee\!\wedge j_1}(s)
\frac{(\mc Z^{i\prime}_s-\mc Z^{j_1}_s)}{\two}\d s\\
&\quad\;+\sum_{j_2:j_2\neq i}\int_0^t a_{i\vee\!\wedge j_2}(s)\frac{(\mc Z^i_s-\mc Z^{j_2}_s)}{\two}\d s+\frac{W^{i\prime}_t-W^i_t}{\two},\quad \forall\;\bi=(i\prime,i)\in \mc E_N,
\end{aligned}
\end{align}
where $i\!\vee\!\!\wedge j=j\!\vee\!\!\wedge i$ is defined below \eqref{def:bijective}.

We now relate all of the summands of
\begin{align}\label{summands:ij}
-\sum_{j_1:j_1\neq i\prime}\int_0^t a_{i\prime\vee\!\wedge j_1}(s)\frac{(\mc Z^{i\prime}_s-\mc Z^{j_1}_s)}{\two}\d s
 +\sum_{j_2:j_2\neq i}\int_0^t a_{i\vee\!\wedge j_2}(s)\frac{(\mc Z^i_s-\mc Z^{j_2}_s)}{\two }\d s
\end{align}
to the nonzero summands of the sum over $\bj\in \mc E_N$ in \eqref{def:tZ} according to the following property: 
\begin{align}
\{\bj\in \mc E_N;\sigma(\bi)\cdot \sigma(\bj)\neq 0\}&=\{\bi\}\cup 
\{\bj\in \mathcal E_N\setminus\{\bi\};\bj\cap \bi\neq \varnothing\}\notag\\
&=\{\bi\}\cup \{\bj\in \mc E_N\setminus\{\bi\};i\prime\in \bj\}\cup \{\bj\in \mc E_N\setminus\{\bi\};i\in \bj\}\notag\\
&=\{\bi\}\cup \{i\prime\!\vee\!\!\wedge j_1;j_1\in \{1,\cdots,N\},j_1\neq i\prime\}\notag\\
&\quad\;\cup \{i\!\vee\!\!\wedge j_2;j_2\in \{1,\cdots,N\},j_2\neq i\}.\label{summands:ij1}
\end{align}
Specifically, we consider the following five cases which are mutually exclusive, collectively exhaustive for $j_1,j_2\in \{1,\cdots,N\}$ with
 $j_1\neq i\prime$ and $j_2\neq i$:\medskip 
 
\noindent {\bf (1)} For $j_1=i$ and $j_2=i\prime$, the sum of the two summands in \eqref{summands:ij} is
\begin{align}
\begin{aligned}\label{eq:dc1}
&\quad \;-\int_0^t a_{i\prime\vee\!\wedge j_1}(s)\frac{(\mc Z^{i\prime}_s-\mc Z^{j_1}_s)}{\two} \d s+\int_0^t a_{i\vee\!\wedge j_2}(s)\frac{(\mc Z^i_s-\mc Z^{j_2}_s)}{\two} \d s\\
&=-2\int_0^t a_\bi(s)\frac{(\mc Z^{i\prime}_s-\mc Z^{i}_s)}{\two}\d s
=-\sigma(\bi)\cdot \sigma(\bi) \int_0^t a_\bi(s)\mc  Z^\bi_s\d s.
\end{aligned}
\end{align}
\noindent {\bf (2)} For $j_1>i\prime$, we have $\min(i\prime,j_1)=i\prime=\max \bi$ so that $\sigma(i\prime\!\vee\!\!\wedge j_1)\cdot \sigma(\bi)=-1$. This gives
\begin{align}
\begin{aligned}\label{eq:dc2}
-\int_0^t a_{i\prime\vee\!\wedge j_1}(s)\frac{(\mc Z^{i\prime}_s-\mc Z^{j_1}_s)}{\two} \d s&=-(-1)\int_0^t a_{i\prime\vee\!\wedge j_1}(s)\frac{(\mc Z^{j_1}_s-\mc Z^{i\prime}_s)}{\two} \d s\\
&=-\sigma( i\prime\!\vee\!\!\wedge j_1)\cdot \sigma(\bi)\int_0^t a_{i\prime\vee\!\wedge j_1}(s)\mc Z^{i\prime\vee\!\wedge j_1}_s\d s.
\end{aligned}
\end{align}
\noindent {\bf (3)} For $i\prime>j_1\neq i$, we have $\max(i\prime,j_1)=i\prime=\max \bi$ so that $\sigma({i\prime\!\vee\!\!\wedge j_1})\cdot \sigma(\bi)=1$. This gives
\begin{align}
\begin{aligned}\label{eq:dc3}
-\int_0^t a_{i\prime\vee\!\wedge j_1}(s)\frac{(\mc Z^{i\prime}_s-\mc Z^{j_1}_s)}{\two}\d s=-\sigma({i\prime\!\vee\!\!\wedge j_1})\cdot \sigma(\bi)\int_0^t a_{{i\prime\vee\!\wedge j_1}}(s)\mc Z^{i\prime\vee\!\wedge j_1}_s\d s.
\end{aligned}
\end{align}
\noindent {\bf (4)} For $i\prime \neq j_2>i$, we have $\min(i,j_2)=i=\min \bi$ so that $\sigma(i\!\vee\!\!\wedge j_2)\cdot \sigma(\bi)=1$. This gives
\begin{align}
\begin{aligned}\label{eq:dc4}
\int_0^t a_{i\vee\!\wedge j_2}(s)\frac{(\mc Z^i_s-\mc Z^{j_2}_s)}{\two} \d s&=-\int_0^t a_{i\vee\!\wedge j_2}(s)\frac{(\mc Z^{j_2}_s-\mc Z^i_s)}{\two}\d s\\
&=-\sigma(i\!\vee\!\!\wedge j_2)\cdot \sigma(\bi)\int_0^t a_{i\vee\!\wedge j_2}(s )\mc Z^{i\vee\!\wedge j_2}_s\d s.
\end{aligned}
\end{align}
\noindent {\bf (5)} For $i>j_2$, we have $\max(i,j_2)=i=\min \bi$ so that $\sigma(i\!\vee\!\!\wedge j_2)\cdot \sigma(\bi)=-1$. This gives
\begin{align}
\begin{aligned}\label{eq:dc5}
\int_0^t a_{i\vee\!\wedge j_2}(s)\frac{(\mc Z^i_s-\mc Z^{j_2}_s)}{\two} \d s=-\sigma(i\!\vee\!\!\wedge j_2)\cdot \sigma(\bi)\int_0^t a_{i\vee\!\wedge j_2}(s)\mc Z^{i\vee\!\wedge j_2}_s\d s.
\end{aligned}
\end{align}

Recall once again that the above five cases are mutually exclusive and collectively exhaustive for $j_1,j_2\in \{1,\cdots,N\}$ with $j_1\neq i\prime$ and $j_2\neq i$, and we have \eqref{summands:ij1}. Hence, the sum of the right-hand sides of the last equalities of these five cases equals the sum over $\bj\in \mc E_N$ in \eqref{def:tZ}. Rewriting \eqref{eq:dc0} with $\mc Z^\bi_t,\mc Z^\bi_0,W^\bi_t$ defined in \eqref{def:bijective} then gives the same equation as \eqref{def:tZ}. We have proved that given a family of processes $\{\mc Z^i_t \}_{1\leq i\leq N}$ from class [$N$], the family of processes $\{\mathcal Z^\bi_t \}_{\bi\in \mc E_N}\cup \{\mc Z^\Sigma_0+W^\Sigma_t \}$ defined by \eqref{def:bijective} is in class [$\mathcal E_N$].

Next, we show that given $\{\mathcal Z^\bi_t \}_{\bi\in \mc E_N}\cup \{\mc Z^\Sigma_0+W^\Sigma_t \}$, \eqref{def:bijective} uniquely defines 
a family
$\{\mc Z^i_t \}_{1\leq i\leq N}$. By using the matrix $\mathsf R$ defined in \eqref{def:R}, all of the equations in 
\eqref{def:bijective}, excluding $a_{ij}(s,\omega)=a_{i\vee\!\wedge j}(s,\omega)$, are equivalent to the following equations:
\begin{align}
[\mc Z^1_t,\cdots,\mc Z^N_t]^\top&\,\defeq\, \mathsf R^{-1}[\mc Z^{(2,1)}_t,\cdots,\mc Z^{(N,N-1)}_t, \mc Z^\Sigma_0+W^\Sigma_t]^\top,\label{def:dc-invZ}\\
[ W^1_t,\cdots, W^N_t]^\top&\,\defeq\, \mathsf R^{-1}[ W^{(2,1)}_t,\cdots, W^{(N,N-1)}_t,W^\Sigma_t]^\top.\label{def:dc-invB}
\end{align}
In the following, we show that
with $a_{ij}\defeq a_{i\vee\!\wedge j}$ as defined in \eqref{def:bijective}, $\{\mc Z^i_t\}_{1\leq i\leq N}$ is in class [$N$]. \medskip 

\noindent {\bf Step 1.} We show that $\{W^i_t \}_{1\leq i\leq N}$ consists of i.i.d. $d$-dimensional standard Brownian motions starting from $0$. By the assumptions on $\{W^\bi_t\}_{\bi\in \mc E_N}\cup \{W^\Sigma_t\}$,
\begin{align}
&\{W^\bi_t\}_{\bi\in \mc E_N}\ind \{W^\Sigma_t\},\\
&\{W^\bi_t \}_{\bi\in \mc E_N}\stackrel{\rm (d)}{=}\{\wt{W}^\bi_t\}_{\bi\in \mc E_N}\,\defeq\,\{(\wt{W}^{i\prime}_t-\wt{W}^i_t)/\two \}_{\bi\in \mc E_N},\label{W:id1}\\
&\{W_t^\Sigma\}\stackrel{\rm (d)}{=}\{\wt{W}^\Sigma_t\}\,\defeq\,\{N^{-1}\textstyle \sum_{i=1}^N\wt{W}^i_t\},
\end{align}
where $\{\wt{W}^i_t \}_{1\leq i\leq N}$ are i.i.d. $d$-dimensional standard Brownian motions starting from $0$ so a fortiori $\{(\wt{W}^{i\prime}_t-\wt{W}^i_t)/\two \}_{ \bi\in \mc E_N}\ind \{N^{-1}\textstyle \sum_{i=1}^N\wt{W}^i_t \}$.  Therefore,
$\{W^\bi_t\}_{\bi\in \mc E_N}\cup \{W^\Sigma_t\}\stackrel{\rm (d)}{=}\{\wt{W}^\bi_t\}_{\bi\in \mc E_N}\cup \{\wt{W}^\Sigma_t\}$.
By the equivalent of \eqref{def:dc-invB} for ``$\wt{W}$,'' the required property of $\{W^i_t\}_{1\leq i\leq N}$ follows. \medskip 
 \medskip 

\noindent {\bf Step 2.}
Next, we show that with probability one, each of the following three processes satisfies difference-consistency at every fixed $t$:
\begin{align}\label{dcfamilies}
\textstyle \{-\sum_{\bj\in \mc E_N}\sigma(\bi)\cdot \sigma(\bj)\int_0^t a_\bj(s)\mc Z^\bj_s\d s\}_{\bi\in \mc E_N},\;\; \{W^\bi_t\}_{\bi\in \mc E_N},\;\; \{\mc Z^\bi_t\}_{\bi\in \mc E_N}.
\end{align}
The difference-consistency in the first process of \eqref{dcfamilies} is a consequence of Lemma~\ref{lem:dc} proven below, and the difference-consistency in the second process of \eqref{dcfamilies} follows from \eqref{W:id1}. Finally, the difference-consistency in the third process of \eqref{dcfamilies} follows upon applying the difference-consistency in the first two processes of \eqref{dcfamilies} and the assumed difference-consistency of $\{\mc Z^\bi_0\}_{\bi\in \mc E_N}$ to  \eqref{def:tZ}, since difference-consistency respects linearity. 
\medskip 

\noindent {\bf Step 3.}
We show that $\{\mc Z^i_t \}_{1\leq i\leq N}$ satisfies \eqref{def:Zi}. First, by using \eqref{def:dc-invZ}, the difference-consistency of $\{\mc Z^\bi_t \}_{\bi\in \mc E_N}$ obtained in Step~2, and the telescoping representations \eqref{eq:telescoping}, we know that $(\mc Z^{i\prime}_t-\mc Z^i_t)/\two=\mc Z_t^\bi$ for all $\bi=(i\prime,i)\in \mc E_N$. Hence, by reversing the arguments in \eqref{eq:dc1}--\eqref{eq:dc5}, \eqref{def:tZ} implies 
\eqref{eq:dc0}. Then we consider $N$ many equations as follows:
\begin{itemize}
\item [\bf (I)] the first $N-1$ equations are given by \eqref{eq:dc0} for $\bi=(2,1),(3,2),\cdots,(N,N-1)$, and
\item [\bf (II)] the $N$-th equation is 
\[
\frac{1}{N}\sum_{i=1}^N\mc Z^i_t=\frac{1}{N}\sum_{i=1}^N\mc Z^i_0-\frac{1}{N}\sum_{i=1}^N \sum_{j:j\neq i}\int_0^t a_{i j}(s)(\mc Z^i_s-\mc Z^j_s)\d s
+\frac{1}{N}\sum_{i=1}^N W^i_t.
\]
\end{itemize}
Note that the $N$-th equation combines the following four equations:
\begin{itemize}
\item  $N^{-1}\sum_{i=1}^N\mc Z^i_t= \mc Z^\Sigma_0+W^\Sigma_t$ by using
 \eqref{def:dc-invZ} and the last row of $\mathsf R$ in \eqref{def:R}. 
 \item $\mc Z^\Sigma_0=N^{-1}\sum_{i=1}^N \mc Z^i_0$ by taking $t=0$ in the previous equation and using the assumption that $W^\Sigma_0=0$.
 \item 
$W^\Sigma_t=N^{-1}\sum_{i=1}^N W^i_t$ by using \eqref{def:dc-invB} and the last row of $\mathsf R$ in \eqref{def:R}. 
\item $N^{-1}\sum_{i=1}^N \sum_{j:j\neq i}\int_0^t a_{i j}(s)(\mc Z^i_s-\mc Z^j_s)\d s=0$ by the symmetry $a_{ij}=a_{ji}$ for all $i\neq j$, since $a_{ij}\defeq a_{i\vee\!\wedge j}$ in \eqref{def:bijective}.
\end{itemize}

Now, observe that both sides of the $N$ equations from {\bf (I)} and {\bf (II)} are images of $\mathsf R$. Solving the $N$ equations by applying $\mathsf R^{-1}$ on both sides gives \eqref{def:Zi}. This property and Step~1 prove that $\{\mc Z^i_t\}_{1\leq i\leq N}$ is in class $[N]$. The proof is complete.
\end{proof}

\vspace{-.5cm}
\begin{lem}\label{lem:dc}
Given an integer $N\geq 3$ and an arbitrary family  $\{p^\bi\}_{\bi\in \mc E_N}$ of $E$-valued points, the family 
$\{q^\bi\}_{\bi\in \mc E_N}$ defined by $q^\bi\,\defeq-\sum_{\bj\in \mc E_N}\sigma(\bi)\cdot\sigma(\bj)p^\bj$ is difference-consistent.
\end{lem}
\begin{proof}
For $\bi,\bj\in \mathcal E_N$ with $i=j\prime$ and for $\bi\oplus\bj$ defined in \eqref{defeq:dc}, $\sigma(\bi)+\sigma(\bj)=\sigma(\bi\oplus \bj)$. Thus, 
\[
q^{\bi}+q^{\bj}=-\sum_{\bk\in \mc E_N}\sigma(\bi)\cdot\sigma(\bk)p^\bk-\sum_{\bk\in \mc E_N}\sigma(\bj)\cdot\sigma(\bk)p^\bk=-\sum_{\bk\in \mc E_N}\sigma(\bi\oplus \bj)\cdot\sigma(\bk)p^\bk=q^{\bi\oplus \bj},
\] 
which gives the required difference-consistency of $\{q^\bi\}_{\bi\in \mc E_N}$ by Definition~\ref{def:dc}. 
\end{proof}

\end{document}